\documentclass[a4paper,12pt,epsfig]{article}
\usepackage{amsmath,amssymb,amsthm,amscd}
\usepackage[pdftex]{graphicx}
\usepackage{wrapfig}
\usepackage{ifpdf}

\usepackage{url}
\newcommand{\conv}{\rm conv}

\newcommand{\ha}{HA}

\newcommand{\rr}{R}
\newcommand{\La}{\rm L}
\newcommand{\as}{\mathcal A}
\newcommand{\gr}{\Gamma}

\newtheorem{thm}{Theorem}
\newtheorem{prop}{Proposition}
\newtheorem{cor}{Corollary}
\newtheorem{lem}{Lemma}

\title {Ramanujan's theorem and highest abundant numbers}

\author {Oleg R. Musin}

\begin{document}
\date{}
\maketitle

\begin{abstract}  In 1915, Ramanujan proved asymptotic inequalities for the sum of divisors function, assuming the Riemann hypothesis (RH). We consider a strong version of Ramanujan's theorem and define highest abundant numbers that are extreme with respect to the Ramanujan and Robin inequalities. Properties of  these numbers are very different depending on whether the RH is true or false.  
\end {abstract}



\medskip

\section{Introduction}

The function $\sigma(n)=\sum_{d|n}{d}$ is the {\em sum of divisors} function.  In 1913  Gr\"onwall (see \cite[Theorem 323]{HW}) 
proved that the asymptotic maximal size of $\sigma(n)$ satisfies
$$
\limsup\limits_{n\to \infty}{G(n)}=e^\gamma, \quad G(n):=\frac{\sigma(n)}{n\log{\log{n}}}, \; n\ge2, 
$$
where $\gamma\approx 0.5772$ is the Euler--Mascheroni constant. Robin \cite{Robin} showed that the Riemann hypothesis (RH) is true if and only if
$$
\sigma(n)<e^\gamma n\log{\log{n}} \; \mbox{ for all } \; n>5040. \eqno (\rr)
$$
 Briggs'  computation of the colossally abundant numbers implies (R) for  $n<10^{(10^{10})}$ \cite{Briggs}.  According to Morrill and Platt \cite{MP18}, (R) holds for all integers  $5040<n<10^{(10^{13})}$.


\medskip

  A positive integer $n$ is called {\em superabundant} (SA) if
$$
\frac{\sigma(k)}{k}< \frac{\sigma(n)}{n}  \; \mbox{ for all integer } \; k \in[1,n-1].
$$
{\em Colossally abundant} numbers  (CA) are those numbers $n$ for which there is $\varepsilon>0$ such that
$$
\frac{\sigma(k)}{k^{1+\varepsilon}}\le \frac{\sigma(n)}{n^{1+\varepsilon}}  \; \mbox{ for all } \; k\in{\mathbb N}.
$$

Bachmann  (see \cite[Theorem 324]{HW})  showed that on average, $\sigma(n)$ is around $\pi^2n/6$. Bachmann and Gr\"onwall's results ensure that for every $\varepsilon>0$ the function ${\sigma(n)}/{n^{1+\varepsilon}}$  has a maximum and that as $\varepsilon$  tends to zero these maxima will increase. Thus there are infinitely many CA numbers.

SA and CA numbers were studied in detail by Alaoglu \& Erd\H{o}s \cite{AE} and Erd\H{o}s \& Nicolas \cite{EN}. The study of numbers with $\sigma(n)$ large  was initiated by Ramanujan \cite{Ram15}. In fact, SA and CA numbers had been studied by Ramanujan in 1915. Unknown to Alaoglu and Erd\"os, about 30 pages of Ramanujan's  paper ``Highly Composite Numbers''' were suppressed. Those pages were finally published in 1997 \cite{Ram97}.

Let 
$$
F (x,k):=\frac{\log(1+1/(x+...+x^k))}{\log{x}},
$$ 
$$
E_p:=\{F(p,k)\,|\, k\ge1\}, \quad p \; \mbox{ is prime}, 
$$
and 
$$
E:=\bigcup\limits_p{E_p}=\{\varepsilon_1,\varepsilon_2,...\}=\left\{\log_2\left(\frac{3}{2}\right),\log_3\left(\frac{4}{3}\right),\log_2\left(\frac{7}{6}\right),...\right\}.
$$
 Alaoglu and Erd\"os \cite[Theorem 10]{AE} showed that if $\varepsilon$ is not {\em critical}, i.e. $\varepsilon\notin E$, then ${\sigma(k)}/{k^{1+\varepsilon}}$ has a unique maximum attained at the number $n_\varepsilon$. If $\varepsilon$ satisfies $\varepsilon_i>\varepsilon>\varepsilon_{i+1}$, $i\in{\mathbb N}$,  then $n_\varepsilon$ is constant on the interval $(\varepsilon_{i+1},\varepsilon_{i})$ and we call it $n_i$. Moreover, 
$$
n_\varepsilon=\prod\limits_{p\in{\mathbb P}}{p^{a_\varepsilon(p)}}, \; \mbox{ where } \; a_\varepsilon(p)=\left\lfloor{\frac{\log(p^{1+\varepsilon}-1)-\log(p^{\varepsilon}-1)}{\log{p}}}\right\rfloor-1. 
$$
In particular, Alaoglu and Erd\H{o}s in their 1944 paper found all SA and CA numbers up to  $10^{18}$. The first 14 CA numbers $n_1, n_2,...,n_{14}$  are 
$$
2, 6, 12, 60, 120, 360, 2520, 5040, 55440, 720720, 1441440, 4324320, 21621600, 367567200. 
$$



 Robin \cite[Sect. 3: Prop. 1]{Robin} showed that if the Riemann hypothesis is false, then there exists a counterexample to the Robin criterion (R) which is a colossally abundant number.  Thus, it suffices to check (R) only for CA numbers.

\medskip 

  Ramanujan, see \cite[p. 143]{Ram97}, proved that if $n$ is a CA number (he called CA numbers as {\em generalized superior highly composite}) then under the RH the following inequalities hold  
$$
\limsup\limits_{n\to\infty}{\left(\frac{\sigma(n)}{n}-e^\gamma\log\log{n}\right)\sqrt{\log{n}}} \le -c_1, \; c_1:=e^\gamma(2\sqrt{2}-4-\gamma+\log{4\pi}) \approx 1.3932, \eqno (1)
$$
$$
\liminf\limits_{n\to\infty}{\left(\frac{\sigma(n)}{n}-e^\gamma\log\log{n}\right)\sqrt{\log{n}}} \ge -c_2,\; c_2:=e^\gamma(2\sqrt{2}+\gamma-\log{4\pi}) \approx 1.5578. \eqno (2)
$$

Denote 
$$
T(n):= {\left(e^\gamma\log\log{n}-\frac{\sigma(n)}{n}\right)\sqrt{\log{n}}}. 
$$
It is easy to see that Ramanujan's inequalities (1) and (2) yield the following fact:

\medskip

\noindent {\em If the RH is true, then there is $i_0$ such that for all CA numbers $n_i, \, i\ge i_0,$ we have} 
$$
1.393< T(n_i)<1.558.  \eqno (3) 
$$

Note that  (2) does not hold for all integers. Indeed, if $p_i$ is prime, then $\sigma(p_i)=p_i+1$. Therefore, $$\limsup\limits_{i\to\infty} {T(p_i)}=\infty.$$
However, (1) holds for all numbers. In Section 2 we prove the following theorem. 

\begin{thm}[The Strong Ramanujan Theorem]\label{ThR} If the RH is true, then
$$
\liminf\limits_{n\to\infty}\,{T(n)} \ge c_1>1.393.
$$
\end{thm}
It is an interesting open problem: {\em Can Ramanujan's constant $c_1$ be improved?}

\medskip

 Theorem \ref{ThR} implies the following inequality (see Corollary \ref{corR} in Section 2): 

\medskip

\noindent{\em If the RH is true, then there is $m_0$ such that for all  $n>m_0$ we have} 
$$
\sigma(n)+\frac{1.393\,n}{\sqrt{\log{n}}} < e^\gamma n\log\log{n} \eqno (4)
$$ 
which is stronger than {\em Ramanujan's theorem} \cite[Theorem 7.2]{Bro}:

\medskip

\noindent{\em If the RH is true, then there is $m_0$ such that  for all  $n>m_0$ we have}
$$
\sigma(n)<e^\gamma n\log{\log{n}}.  \eqno (5)
$$



\medskip

Note that, for fixed $\varepsilon>0$,  
 CA numbers $n$ may be viewed as maximizers of 
 $$
 Q(k)-\varepsilon \log{k}=\log({\sigma(k)}/{k^{1+\varepsilon}}), \quad Q(k):=\log{\sigma(k)}-\log{k}.
 $$ 
 Equivalently,  $n$ is CA if $(x_n,A(x_n))$ is a {\em vertex of the convex envelope of $A$ on $D$,}  where 
 $$x_k:=\log{k},  \quad A(x_k):=x_k-\log{\sigma(k)}=-Q(k), \quad D:=\{x_k\}, \quad k\ge2,$$ 
see details in  Section 3, Example 1. 
  
\medskip  

Let  $n\ge2$ and $s$ be a real number. Denote 
$$
\rr_s(n):=\left(e^\gamma n\log\log{n}-\sigma(n)\right)(\log{n})^s.  
$$

Now we define {\em Highest Abundant} (HA) numbers.  {\em We say that  $n\in D\subset{\mathbb N}$ is $\ha$ with respect to $\rr_s$ and write $n\in\ha_s(D)$ if for some real $a$ 
$$
\rr_s(k)-ak
$$
attains its minimum on $D$ at $n$. For  $D=\{n\in{\mathbb N}\,|\, n\ge 5040\}$ we denote $\ha_s(D)$ by $\ha_s$. 
}

\medskip

Actually, if $D$ is infinite, then  $\ha_s(D)$ can be empty or contain only one number $m_0$. It is clear that $m_0$ is the minimum number in $D=\{m_0=x_0, x_1,...\}$.  Then   there is $a_0$ such that $m_0$ is defined by any $a\le a_0$. 

It can be shown, see Proposition \ref{prop1} in Section 3, if $\ha_s(D)=\{m_0,m_1...\}$ contains at least two numbers, then there is a set of {\em critical} values $a$, $\as_s(D)=\{a_1,a_2,...\}$, $a_1<a_2<...$, such that if $a$ is not critical, then  $\rr_s(n)-na$ has a unique minimum on $D$ attained at the number $m_a$. If $a\in (a_i,a_{i+1})$, $i=1,2,...$, then $m_a$  is constant on the interval   $(a_i,a_{i+1})$ and $m_a=m_i$. In fact, $a_i$ is the slope of $\rr_s$ on $[m_{i-1},m_i]$, i.e. 
$$
a_i=\frac{\rr_s(m_i)-\rr_s(m_{i-1})}{m_i-m_{i-1}}. 
$$


%
We see that definitions of CA and HA numbers are similar, in both cases numbers can be determined through the vertices  of the convex envelopes of certain functions. In Example 3 (Section 3) is considered HA numbers with respect to $\rr_s$, $s=1$, on $D=[2,n_{13}=21621600]$. There are 13 HA numbers in this interval, 12 of them are CA numbers (except $n_6=360$) and one more $m=2162160$ is SA but $m$ is not CA.
However,  properties of HA and CA numbers are different. The property that $\ha_s$  is infinite  depending on whether the RH is true or false.   


\begin{thm}\label{Th1} (i) Let $s>1/2$. If the RH is true, then $\ha_s$  is infinite  and $\lim\limits_{n\to\infty}{a_n}=\infty$. \\
If the RH is false, then $\ha_s$  is empty. 

\medskip

\noindent (ii) Let $s\le0$. If the RH is false, then $\ha_s$  is infinite, all $a_i<0$ and $\lim\limits_{n\to\infty}{a_n}=0.$ \\
If the RH is true, then $\ha_s=\{5040\}$ and   $\as_s=\{0\}$. 
\end{thm}

In Section 4 (Theorems \ref{ThM1} and \ref{ThM2}) we consider extensions of Theorem \ref{Th1}. Proofs of these theorems rely on Robin's inequalities (7) and (8)  [Section 4], the strong Ramanujan theorem and his inequality (2), namely on Corollary \ref{cor2} in Section 2.

\medskip

Let $h_n:=\sum_{i=1}^n{1/i}$ denote the harmonic sum. Using (R) Lagarias \cite{Lagarias} showed that the Riemann hypothesis is equivalent to the following inequality 
$$
L_0(n):= h_n+\exp(h_n)\log(h_n)-\sigma(n)>0 \; \mbox{ for all } \; n>1.  \eqno (\La)
$$
In Section 4 we consider an analog of Theorem \ref{Th1} for $(\La)$ on $D={\mathbb N}$.

\section{The strong Ramanujan theorem}
Ramanujan's theorem in the form of (5) is present in \cite[Theorem 7.2]{Bro}, \cite{NS14}, \cite[p. 152]{Ram97} and other places. This theorem can be easily derived from (1) for the  CA numbers. Theorem \ref{ThR} extends (1) for all $n\in{\mathbb N}$ and is a strong version of Ramanujan's theorem, see (4). However, we could not find a proof of Theorem \ref{ThR} for arbitrary integers.  In this section we fill this gap.


\begin{proof}[Proof of Theorem \ref{ThR}]   Let
$$
f(n):=\sqrt{\log{n}}\,\log\log{n}, \quad g(n):=e^\gamma - G(n). 
$$
Then $T(n)= f(n)\,g(n).$

  Let $S$ be the set of all non--CA integers $n>2$.  Then for every $n\in S$ there is $i=i(n)>1$ such that $n_{i-1}<n<n_{i}$, where   $n_{i-1}$ and $n_{i}$  are two consecutive CA numbers. Robin \cite[Proposition 1]{Robin} showed that
 $$G(n)\le \max(G(n_{i-1}),G(n_i)).$$ 
We divide $S$ into two disjoint subsets $S_1$ and $S_2$:
$$
S_1:=\{n\in S\,|\,G(n)\le G(n_{i-1})\}, \quad S_2:=\{n\in S\,|\,G(n_{i-1})< G(n)\le  G(n_i)\}.
$$ 

\noindent (1) Suppose $n\in S_1$. Then $g(n)\ge g(n_{i-1})$, where  $i=i(n)$.  Since $f$ is a monotonically increasing function,  we have $f(n)>f(n_{i-1})$ and $T(n)>T(n_{i-1})$. Thus, 
 $$
\liminf\limits_{n\in S_1, n\to\infty} {T(n)} \ge \liminf\limits_{i\to\infty}{T(n_{i-1})}=\liminf\limits_{i\to\infty}{T(n_i)} \ge c_1. 
$$
 
 \noindent (2) Suppose $n\in S_2$. Then $g(n)\ge g(n_{i})$ and $f(n)>f(n_{i-1})$. That yields 
 $$
 T(n)>f(n_{i-1})g(n_i)=T(n_i)F(i), \quad F(i):=\frac{f(n_{i-1})}{f(n_{i})}. 
 $$
 
We have
 $$
 \lim\limits_{i\to\infty}{\frac{\log(n_{i-1})}{\log(n_{i})}}=1. \eqno (6) 
 $$
 Indeed, let $P(n)$ denote the largest prime factor of $n$.  Alaoglu \& Erd\H{o}s \cite[Theorem 7]{AE}  proved that $P(n) \sim \log{n}$ for all SA numbers. Then, in particular, it holds for CA numbers. The quotient of two consecutive CA numbers is either a prime or the product of two distinct primes \cite[page 455]{AE}, \cite[Lemma 6.15]{Bro}, i.e.  $n_{i}\le n_{i-1}P^2(n_i) \sim n_{i-1}\log^2(n_{i})$. Then we have 
 $$
1> \frac{\log(n_{i-1})}{\log(n_i)} > \frac{\log(n_i)-2\log(P(n_i))}{\log(n_i)} \sim 1-\frac{2 \log \log{n_i}}{\log n_i} \sim 1.
 $$

It is not hard to see that (6) implies  
$
 \lim\limits_{i\to\infty}{F(i)}=1. 
$
That yields
$$
\liminf\limits_{n\in S_2, n\to\infty} {T(n)} \ge \liminf\limits_{i\to\infty}{T(n_i)F(i)}=\liminf\limits_{i\to\infty}{T(n_i)} \ge c_1. 
$$

Thus, we have (1) for CA, $S_1$ and $S_2$, i.e. for all numbers.  
\end{proof}

\noindent{\bf Remark.} In the first version of this paper our proof of Case (2) relies on  \cite[Theorem 1]{Wu19}. I am very grateful to Xiaolong Wu for the idea of proving this case using (6). Note that (6) is easily derived from the results of the classical paper of Alaoglu and Erd\H{o}s \cite{AE}.  

\begin{cor}\label{corR} If the RH is true, then for every $\varepsilon>0$ there is $m_0$ such that for all  $n>m_0$ we have
$$
\sigma(n)+ (c_1-\varepsilon)\frac{n}{\sqrt{\log{n}}} < e^\gamma n\log\log{n}. 
$$ 
In particular, if\, $\varepsilon\le1.393$, then
$
\sigma(n)<e^\gamma n\log{\log{n}} \; \mbox{ for all } \; n>m_0. 
$
\end{cor}

From  (2)  for CA numbers $n_i$ we have
$$\limsup\limits_{i\to\infty}\,{T(n_i)}\le c_2<1.558.$$ 
This fact and Corollary \ref{corR} yield the following corollary: 

\begin{cor}\label{cor2} If the RH is true, then for every $\varepsilon>0$ there is $m_0$ such that a set 
$$
M(\varepsilon):=\{n>m_0\,|\,T(n)<c_2+\varepsilon\}
$$ 
is infinite and for all $n\in M(\varepsilon)$ we have $T(n)>c_1-\varepsilon.$
\end{cor}

\section{Convex envelope of functions}

Let $D=\{x_n\}$ be an increasing sequence. Let $h:D\to  {\mathbb R}$ be a function on $D$. We say that $h$ is {\em convex} (or {\em concave upward}\,)  on $D$ if for all  $a,x,b\in D$ such that $a< x< b$ we have
$$
h(x)\le\frac{(b-x)h(a)+(x-a)h(b)}{b-a}. 
$$

{Denote by $\Omega(f)$ the set of all convex functions $h:D\to  {\mathbb R}$ such that $h(x)\le f(x)$ for all $x\in D$.}  Suppose $\Omega(f)\ne\emptyset$. The {\em lower convex envelope} ${\breve  f}$ of a function $f$  on $D$  is defined at each point of $D$ as the supremum of all convex functions that lie under that function, i.e.
$$
\breve  f(x):=\sup\{h(x)\,|\,h\in\Omega(f)\}.
$$


\medskip 

Alternatively, $\breve f$ can be defined as follows. Let 
$$
 \gr_f:=\{(x,f(x))\in D\times{\mathbb R}\subset{\mathbb R}^2\}
$$
be the graph of $f$. The {convex hull} of $\gr_f$ in ${\mathbb R}^2$ is the set of all convex combinations of points in $\gr_f$:
$$
\conv(\gr_f):=\{c_1p_1+...+c_kp_k\,|\,p_i\in G_f, c_i\ge0, i=1,...,k, c_1+...+c_k=1\}.
$$
Then the graph
$
 \{(x,\breve  f(x))\in D\times{\mathbb R}\}
$
is the {\em lower convex hull} of $\conv(\gr_f)$.

\medskip 

It is clear, if $D$ is finite,  then $\Omega(f)$  is not empty.  However,  if $D$ is infinite, then $\Omega(f)$ can be empty, for instance if  $f(n)=-n^2$ and $D={\mathbb N}$.  

Let $f$ be a function on $D=\{x_0,x_1,...\}$ with  $\Omega(f)\ne\emptyset$. Then $\breve  f$ is a piecewise linear convex function on $D$. Hence, there is a subset 
$$
H_f:=\{m_0=x_0,m_1,...\}\subset D$$ 
such that $\breve  f$ is a linear function on $[m_{i-1},m_i]$, $\breve  f(m_i)=f(m_i)$ for all $i$,  and the sequence of slopes $A_f:=\{a_1,a_2,...\}$ is strictly monotonic increasing, i.e. $a_1<a_2<...$, where
$$
a_i:=\frac{f(m_i)-f(m_{i-1})}{m_i-m_{i-1}}. 
$$

\begin{figure}
\begin{center}

 \includegraphics[trim=0 7.5cm 0 7cm, clip,scale=.7]{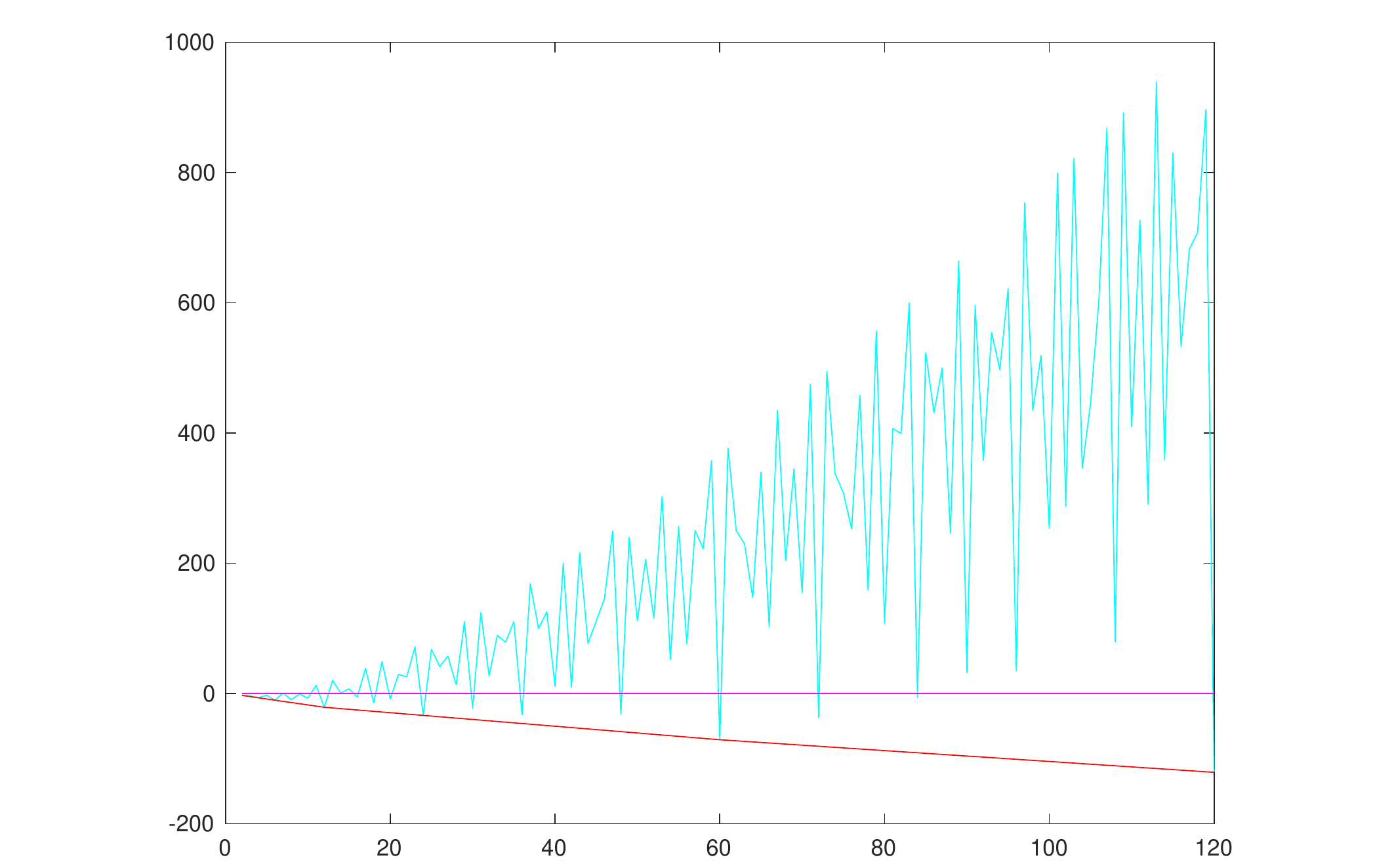}

\end{center}
\caption{Graphs of $\rr_1$ and   $\breve\rr_1$ on $D=\{2,...,120\}$.}
\end{figure}

\medskip

Let $\tilde H_f$ be a subset in $D$ such that $m\in\tilde H_f$ if for some $a\in{\mathbb R}$ the function 
$
f(x)-ax
$
attains its minimum on $D$ at $m$, i.e.
$$
 \tilde H_f:=\{m\in D\,|\,\exists  a\in{\mathbb R},\,  \forall x\in D, \,  f(m)-ma\le f(x)-ax\}.
$$

The next proposition can be easily derived from the above definitions. 
\begin{prop}\label{prop1} Let $f$ be a function on $D=\{x_n\}$ with  $\Omega(f)\ne\emptyset$.   Then $\tilde H_f$ coincides with $H_f$ and every $m_i\in H_f$, $i\ge1$, is uniquely determined by any $a\in (a_{i-1}, a_{i})$. 
\end{prop}

\noindent{\bf Example 1.}  Let $D:=\{x_n\}$, where $x_n:=\log{n}$,  $n\in{\mathbb N}$. Let $f(x_n):=x_n-\log{\sigma(n)}$. Then $f(x_n)=-Q(n)$, where $Q$ is defined in Section 1.  It is easy to see that in this case Proposition \ref{prop1} yields that $H_f$ is the set of CA numbers and $A_f=\{-\varepsilon_i\}$.


If $D=\{x_0,x_1,...,x_l\}$ is finite, then $H_f:=\{m_0=x_0,m_1,...,m_k\le x_l\}$ and the cardinality $|A_f|=k$. If $D$ is infinite, then $A_f$ can be (i) infinite or (ii) finite.   It is not hard to see that in case (ii) $H_f:=\{m_0,m_1,...,m_k\}$ and $A_f=\{a_1,...,a_k,a_{k+1}\}$, where 
$$
\breve  f(n)=\breve  f(m_k)+(n-m_k)\,a_{k+1} \; \mbox{ for all } \; n\in D, \; n\ge m_k. 
$$
Let $m_{k+1}:=\infty$. Then for both cases we have  that  $a_i$ is the slope of $\breve  f$ on $[m_{i-1},m_i]$.

\medskip

\noindent{\bf Example 2.} Let $f(n)=\rr_1(n)=\left(e^\gamma n\log\log{n}-\sigma(n)\right)\log{n}$  on  $D=\{2,...,120\}$. Then  $H_f=\{2,6,12,60,120\}$. (Note that $H_f$ consists of the first five CA numbers.)  In this case $\breve\rr_1$ is a convex monotonically decreasing function on $D$, see Figure 1. 

\medskip

\noindent{\bf Example 3.}  Let $f=\rr_1$ 
on $D=\{2,3,...,n_{13}=21621600\}$. Then 
$$
H_f=\{2, 6, 12, 60,1 20, 2520, 5040, 55440, 720720, 1441440, 2162160, 4324320,  21621600\}.
$$
In this list of 13 numbers $m_0,...,m_{12}$ there are 12 out of the first 13 CA numbers except $n_6=360$. However, $m_{10}$ is an SA  number $2162160=2^4\cdot 3^3\cdot 5\cdot 7\cdot 11\cdot 13$ but is not CA. $\breve\rr_1$ on $H_f$ has a minimum at $m_5=2520$ and is positive for $m_i>m_6=5040$. We have
$$
a_1<...<a_5<0<a_6<...<a_{12}. 
$$

Now we prove the main results of this section. 

\begin{lem}\label{L3} Let $D\subset{\mathbb N}$ be infinite and $n_0\in D$.  Let  $f$ and  $g$ be functions on $D$ such that 
$$
f(n)\ge g(n) \; \mbox{ for all } \; n\ge n_0 \;  \mbox{ and } \ \lim\limits_{n\to\infty}{\frac{g(n)}{n}}=\infty. 
$$
Then $A_f$ is infinite and $\lim\limits_{n\to\infty}{a_n}=\infty.$
\end{lem}

\begin{proof} By assumption, for any real $a$ there is $n_a$ such that $g(n)>an$ for all $n\ge n_a$. This fact yields that for any linear function $l(x)=ax+b$  there is no or there are finitely many $n\in D$ such that $f(n)\le l(n)$. 

Let $D=\{x_0,x_1,...\}$, $m_0:=x_0$ and $H_f^{(0)}:=\{m_0\}$. Suppose  $H_f^{(i)}=\{m_0,\ldots,m_i\}$ and $\as_f^{(i)}=\{a_1,\ldots,a_i\}$.  Let $l(x)$ be a linear function given by two points $(m_i,f(m_i))$ and $(x_{k+1},f(x_{k+1}))$, where $m_i=x_k$. Denote 
$$D_l=\{n\in D\,|\, n>m_i, f(n)\le l(n)\}.$$
 We have $1\le |D_l|<\infty$. Let $x_j$ be a number in $D_l$ such that the slope of  a linear function  given by two points $(m_i,f(m_i))$ and $(n,f(n))$, $n\in D_l$, attains its minimum at $x_j$. We denote the correspondent linear function by $l_{i+1}$. It is clear that $f(n)\ge l_{i+1}(n)$ for all $n\in D_l$. Hence, $m_{i+1}=x_j$ and $a_{i+1}$ is the slope of $l_{i+1}$. We can continue this process. Since $f(n)/n\to\infty$ as $n\to\infty$, we have $a_i\to\infty$ as $i\to\infty$. 
\end{proof}

\medskip

\begin{lem}\label{L4}  Let  $g_1$ and  $g_2$ be functions on  $D\subset{\mathbb N}$ such that for all $n\in D$ we have
$$
g_2(n)\ge g_1(n)   \; \mbox{ and } \; \lim\limits_{n\to\infty}{g_1(n)}=\infty, \quad        \lim\limits_{n\to\infty}{\frac{g_2(n)}{n}}=0.
$$
Suppose for a function $f$ on $D$ there is $n_0\in D$ such that $f(n)\ge g_1(n)$ for all $n\ge n_0$.   If there are  infinitely many $n\in D$ such that $f(n)\le g_2(n)$, then $A_f=\{a_1,...,a_k\}$ is finite and 
$$
a_1<...<a_k=0.
$$
\end{lem}

\begin{proof} Denote $D_0:=\{n\in D\,|\,n<n_0\}$ and $D_1:=\{n\in D\,|\,n\ge n_0, f(n)\le g_2(n)\}$   By assumption, $D_1$ is infinite and for any linear function $l(x)=ax+b$ with $a>0$ there is no or there are finitely many $n\in D_1$ such that $f(n)\ge l(n)$.  Hence, all $a_i\le 0$.  Since $f(n)\to \infty$ as $n\to\infty$, we have that $A_f$ is finite and the largest $a_k=0$. 
\end{proof}

\begin{lem}\label{L5}  Let  $g_1$ and  $g_2$ be functions on  $D\subset{\mathbb N}$ such that for all $n\in D$ we have
$$
g_2(n)\ge g_1(n)  \; \mbox{ and } \; \lim\limits_{n\to\infty}{g_2(n)}=-\infty, \quad        \lim\limits_{n\to\infty}{\frac{g_1(n)}{n}}=0.
$$
Suppose for a function $f$ on $D$ there is $n_0\in D$ such that $f(n)\ge g_1(n)$ for all $n\ge n_0$.   If there are  infinitely many $n\in D$ such that $f(n)\le g_2(n)$, then $A_f$ is infinite and $\lim\limits_{n\to\infty}{a_n}=0.$
\end{lem}

\begin{proof} It is not hard to see that the assumptions yield that for any  $l(x)=ax+b$ with $a<0$ there is no or there are finitely many $n\in D$ such that $f(n)\le l(n)$. Let $l_i$ be the same as in Lemma \ref{L3}. 
In this case for $n\in D_1$, that defined in Lemma \ref{L4}, we have $f(n)\to-\infty$ and $f(n)/n\to 0$ as $n\to\infty$. Thus, $a_i\to0$ as $i\to\infty$. 
\end{proof}

\begin{lem}\label{L6} Let $D\subset{\mathbb N}$ be infinite.  Let   $g$ be a function on $D$ such that 
$$
 \lim\limits_{n\to\infty}{\frac{g(n)}{n}}=-\infty. 
$$
Suppose for a function $f$ on $D$ there are  infinitely many $n\in D$ such that $f(n)\le g(n)$. 
Then $\Omega(f)$ is empty. 
\end{lem}
\begin{proof}   Let $D_g:=\{n\in D\,|\,f(n)\le g(n)\}$.   Let $l_n$ be a linear function given by two points $(x_0,f(x_0))$ and $(n,f(n))$. By assumption for any $a$ there is $n\in D_g$ such that the slope of $l_n$ is less than $a$. Moreover, there are infinitely many $m$ in $D_g$ with $f(m)<l_n(m)$. This completes the proof. 
\end{proof}

\section{Proof of Theorem \ref{Th1} and its extensions}

Robin \cite[Theorem 2]{Robin} showed that for all $n\ge 3$ 
$$
\rr_0(n)=e^\gamma n\log{\log{n}}-\sigma(n)>-0.6482\frac{n}{\log{\log{n}}}. \eqno (7)
$$
If the RH is false Robin \cite[Theorem 1]{Robin}  proved that  there exist constants $b\in(0,1/2)$ and $c>0$ such that 
$$
\rr_0(n)<-\frac{c\,n\log{\log{n}}}{(\log{n})^b}  \eqno (8)
$$
holds for infinitely many $n$. Thus, if the RH is false there are infinitely many $n\in{\mathbb N}$ such that 
$$
C_1(n):=-\frac{0.6482\,n}{\log{\log{n}}}<\rr_0(n)<C_2(n):=-\frac{c\,n\log{\log{n}}}{(\log{n})^b}.  
$$

\medskip 

 Let $\tau(n)$ be any positive function on $D\subset{\mathbb N}$. Denote 
$$
\rr_\tau(n):=\left(e^\gamma n\log\log{n}-\sigma(n)\right)\tau(n), \; n\in D.  
$$
We defined HA numbers with respect to $\rr_\tau$ as follows:
$$
\ha_\tau(D):=H_{\rr_\tau}(D)=\{m\in D\,|\,\exists  a\in{\mathbb R},\,  \forall x\in D, \,  \rr_\tau(m)-ma\le \rr_\tau(x)-ax\}.
$$
 As above,  $\as_\tau(D)=\{a_1,a_2,...\}$ are slopes of $\rr_\tau$ on $\ha_\tau(D)$ and  we denote $\ha_\tau(D)$ by $\ha_\tau$ for $D=\{n\in{\mathbb N}\,|\, n\ge 5040\}$.

\medskip

The following theorem extends Theorem \ref{Th1}(i). 
\begin{thm}\label{ThM1} Let $\tau(n)>0$ for all $n\ge 5040$. Denote 
$$
\Phi_\tau:= \lim\limits_{n\to\infty}{\frac{\tau(n)}{\sqrt{\log{n}}}}. 
$$
(a) Assume the RH is true. If $\Phi_\tau=\infty$, 
then $\ha_\tau$  is infinite  and $\lim\limits_{n\to\infty}{a_n}=\infty$. \\
(b) If the RH is false and $\Phi_\tau>0$, then $\ha_\tau$  is empty. 
\end{thm}

\begin{proof} (a) Suppose the RH is true. Let 
$$
g(n):=\frac{1.393\,n\,\tau(n)}{\sqrt{\log{n}}}.
$$
By Corollary \ref{corR} there is $n_0$ such that for all $n\ge n_0$ we have
$$
\rr_\tau(n)=\frac{nT(n)\tau(n)}{\sqrt{\log{n}}}\ge g(n) \; \mbox{ and by assumption } \;  \lim\limits_{n\to\infty}{\frac{g(n)}{n}}=1.393\Phi_\tau=\infty. 
$$
Then Lemma \ref{L3} with $f=\rr_\tau$ yields that $\lim\limits_{n\to\infty}{a_n}=\infty$.

\medskip

\noindent (b)  Suppose the RH is false. Since $b<1/2$ by (8) there are infinitely many $n\in{\mathbb N}$ such that  
$$
\rr_\tau(n)= \rr_0(n)\tau(n)\le C_2(n)\tau(n) < g(n):=-\frac{c\,n\,\tau(n)\log{\log{n}}}{\sqrt{\log{n}}}.
$$
Then Lemma \ref{L6} with $f=\rr_\tau$ completes the proof. 
\end{proof}

Now we consider a generalization of Theorem \ref{Th1}(ii). 

\begin{thm}\label{ThM2} Let 
$$
\tau(n)>0, \, n\ge 5040, \quad \lim\limits_{n\to\infty}{\frac{\tau(n)}{\log{\log{n}}}}=0 \; \mbox{ and } \; \lim\limits_{n\to\infty}{\frac{\tau(n)\,n\log{\log{n}}}{\sqrt{\log{n}}}}=\infty. 
$$
(a) If the RH is false, then $\ha_\tau$  is infinite, all $a_i<0$ and $\lim\limits_{n\to\infty}{a_n}=0.$ \\
(b) If the RH is true, then $\ha_\tau=\{5040\}$ and   $\as_\tau=\{0\}$. 
\end{thm}

\begin{proof} (a) Suppose the RH is false. Let
$$
g_1(n):=C_1(n)\tau(n), \quad g_2(n):=C_2(n)\tau(n).  
$$
Then by (7) we have that $g_1(n)<\rr_\tau(n)$  for all $n\in D$ and by (8) the inequality $\rr_\tau(n)< g_2(n)$ holds for infinitely many $n$. Since $f=\rr_\tau$, $g_1$ and $g_2$ satisfy the assumption of Lemma \ref{L5} we have (a). 

\medskip 

(b)  Suppose the RH is true. Let
$$
g_1(n):=\frac{1.393\,n\,\tau(n)}{\sqrt{\log{n}}}, \quad g_2(n):=\frac{1.558\,n\,\tau(n)}{\sqrt{\log{n}}}.  
$$
Then Corollary \ref{cor2} yields that $f=\rr_\tau$, $g_1$ and $g_2$ satisfy the assumption of Lemma \ref{L4}. Since for all $n>5040$ we have 
$\rr_\tau(n)>0>\rr_\tau(5040)$, there are not $a_i\le0$. Thus,  $\ha_\tau=\{5040\}$. 
\end{proof} 

\medskip

\begin{proof}[Proof of Theorem \ref{Th1}] This theorem immediately follows from Theorems \ref{ThM1} and  \ref{ThM2}.  
Indeed, if $\tau(n)=(\log{n})^s$, then $\rr_\tau(n)=\rr_s(n)$. It clear that $\Phi_\tau=\infty$ in Theorem  \ref{ThM1} only if $s>1/2$ and the assumptions in Theorem \ref{ThM2} hold if $s\le0$. 
\end{proof}








\medskip

From the Lagarias inequalities \cite[Lemmas 3.1, 3.2]{Lagarias} for $n>20$ we have 
$$
R_0(n)+h_n\le L_0(n)\le R_0(n)+\frac{7n}{\log{n}}. \eqno (9)
$$
Let $L_\tau(n):=L_0(n)\tau(n)$. Then (9) yields analogs of Theorems \ref{ThM1} and  \ref{ThM2} for $L_\tau$. We can just substitute $\rr_\tau$ by $L_\tau$. 

\begin{thm}\label{ThL} (i) 
If the RH is true, $\tau(n)>0$ and $\Phi_\tau=\infty$, then there are infinitely many $\ha$ numbers with respect to $L_\tau$ and $\lim\limits_{n\to\infty}{a_n}=\infty$. . \\
 (ii) Let $\tau(n)$ satisfy the assumptions of Theorem \ref{ThM2}. If the RH is false, then there are infinitely many $\ha$ numbers with respect to $L_\tau$, all $a_i<0$ and $\lim\limits_{n\to\infty}{a_n}=0.$ 

\end{thm}

\medskip

\medskip

\medskip

\medskip

O. R. Musin, School of Mathematical and Statistical Sciences, University of Texas Rio Grande Valley,  One West University Boulevard, Brownsville, TX, 78520.

 {\it E-mail address:} omusin@gmail.com

\end{document}